\documentclass[12pt]{article}
\usepackage{amssymb,amsmath,amsthm,amsfonts}
\usepackage{graphicx}
\graphicspath{{./pics/}}
\usepackage{subfig}
\usepackage{float}
\usepackage[colorlinks,linktocpage,linkcolor=blue]{hyperref}
\usepackage{algorithm,algorithmic}

\newtheorem{thm}{Theorem}[section]
\newtheorem{rmk}{Remark}[section]

\newtheorem{exam}{Example}

\title{Gradient-based estimation of Manning's friction coefficient from noisy data}
\author{Victor M. Calo\footnote{Applied Mathematics \& Computational Science and Earth
Science \& Engineering, King Abdullah University of Science and Technology, Thuwal, Saudi Arabia
(victor.calo@kaust.edu.sa, nathaniel.collier@kaust.edu.sa, hany.radwan@kaust.edu.sa)}
\and Nathan Collier\footnotemark[1]
\and Matthias Gehre\footnote{Center for Industrial Mathematics and Department of Mathematics,
University of Bremen, Bremen 28359, Germany (mgehre@math.uni-bremen.de)}\and Bangti Jin\footnote{Institute for
Applied Mathematics and Computational Science and Department of Mathematics, Texas A\&M University, College
Station, Texas 77843-3368, USA (btjin@math.tamu.edu)} \and Hany Radwan\footnotemark[1]}
\date{\today}

\begin{document}
\maketitle

\begin{abstract}
We study the numerical recovery of Manning's roughness coefficient for the
diffusive wave approximation of the shallow water equation. We
describe a conjugate gradient method for the numerical
inversion. Numerical results for one-dimensional model are presented
to illustrate the feasibility of the approach. Also
we provide a proof of the differentiability of the weak form
with respect to the coefficient as well as the continuity and boundedness of the linearized operator
under reasonable assumptions using the maximal parabolic regularity
theory.  
\end{abstract}

\section{Introduction}

The diffusive wave approximation (DSW) of the shallow water equations (SWE)
is often used to model overland flows such as floods, dam breaks, and
flows through vegetated areas \cite{XanthopoulosKoutitas:1976,HromadkaBerenbrock:1985,FengMolz:1997}. 
The SWE result from the full
Navier-Stokes system with the assumption that the vertical momentum
scales are small relative to those of the horizontal momentum. This
assumption reduces the vertical momentum equation to a hydrostatic
pressure relation, which is integrated in the vertical direction to
arrive at a two-dimensional system known as the SWE. The DSW further simplifies the SWE by
assuming that the horizontal momentum can be linked to the water
height by an empirical formula, such as Manning's formula (also known
as Gauckler-Manning formula~\cite{Gauckler:1867})~\cite{Chow:1988,
  Yen:1992}. The DSW is a scalar parabolic equation which resembles
nonlinear diffusion.

The DSW gives rise to the following initial/boundary value problem for
the water height $u$
\begin{equation}\label{eqn:dsw}
  \left\{\begin{array}{ll}
      \dfrac{\partial u}{\partial t}-\nabla\cdot
      \left(k\left(u,\nabla u\right)\nabla u\right) =f
      & \mbox{in } \Omega\times(0,T]\\
        u= u_0 & \mbox{on } \Omega\times\{t=0\}\\
        \left(k\left(u,\nabla u\right)\nabla u\right) \cdot n = h
        &\mbox{on }\Gamma_N\times(0,T]\\
      u = g& \mbox{on } \Gamma_D\times(0,T]
  \end{array}\right.
\end{equation}
where $\Omega$ is an open bounded domain in $\mathbb{R}^d\ (d=1,2)$,
and $\Gamma_N$ and $\Gamma_D$ are disjoint subsets of the boundary
$\Gamma=\partial\Omega$ such that $\Gamma=\Gamma_N\cup\Gamma_D$. The
forcing function (e.g., rainfall acting as a source or infiltration
acting as a sink) $f:\Omega\times(0,T]\rightarrow\mathbb{R}$, the
initial condition $u_0:\Omega\rightarrow \mathbb{R}$, and the Neumann
and Dirichlet boundary conditions $h:\Gamma_N\times(0,T]\rightarrow
\mathbb{R}$ and $g:\Gamma_D\times(0,T]\rightarrow\mathbb{R}$ are
given. The diffusion coefficient $k(u,\nabla u)$ is given by
\begin{equation*}
  k(u,\nabla u)=\dfrac{1}{c_f}\dfrac{\left(u-z\right)^\alpha}{\left|\nabla{}u\right|^{1-\gamma}}=d_f\dfrac{\left(u-z\right)^\alpha}{\left|\nabla u\right|^{1-\gamma}},
\end{equation*}
where $z:\overline{\Omega}:\rightarrow \mathbb{R}^+$ is a nonnegative
time-independent function that represents the bathymetric or
topographic measurements available for the region under analysis. The
parameters $\gamma$ and $\alpha$ satisfy $0<\gamma\leq1$ and
$1<\alpha<2$. Following Manning's
formula~\cite{SantillanaDawson:2010}, we set these parameters to
$\gamma=\frac{1}{2}$ and $\alpha=\frac{5}{3}$. The
function $c_f$ (or equivalently $d_f=\frac{1}{c_f}$) represents
Manning's roughness coefficient, also known as the friction
coefficient. The typical values are available in the
literature~\cite{ArcementSchneider:1990,Walkowiak:2006}. We refer
to~\cite{AlonsoSantillanaDawson:2008, SantillanaDawson:2010} for recent
mathematical analysis and
to~\cite{SantillanaDawson:2010, ColierRadwanDalcinCalo:2011} for
efficient numerical algorithms.

In practice, the Manning coefficient $c_f$ is an empirically derived
coefficient, and historically it was expected to be constant and a
function of the roughness only. It is now widely accepted that the
values of the coefficients $c_f$ are only constant within some range
of flow rates, and depends strongly on many factors, including surface
roughness, sinuosity and flow reach. The presence of multiple
influencing factors renders a direct measurement of the coefficient
values less reliable and the use of a single-valued coefficient also
greatly constrains the practical utility of the DSW model to
faithfully capture important physical features of real open channel
flows, for which a spatially-varying coefficient is necessary due to
distinct physical characteristics of different regions.

In this study, we propose to estimate the distributed Manning
coefficient directly from water height measurements using inversion
techniques, that is, formulating an inverse problem for identifying
the friction coefficient $c_f$ from measurements of the water-height
acquired by sensors and infrared imaging. In comparison with direct
measurement, the proposed approach does not require a knowledge of the
physical properties of the overland environment, which might be
difficult to directly incorporate, and moreover, can naturally handle
spatially varying coefficients. Therefore, a reliable and efficient
estimate of this coefficient is expected to greatly broaden the scope
of the DSW model and to facilitate real-time simulation of the flow,
which is of immense significance in a number of
applications, for example flood prediction and flood hazard
assessment. The goal of the present study is to propose an inversion
algorithm and demonstrate its feasibility on simulation data for
one-dimensional models.

We briefly comment on relevant studies on the inverse problem. Due to
its conceived practical significance, it has received some attention
in the literature~\cite{DingJiaWang:2004,DingWang:2005}. For example,
Ding et al~\cite{DingJiaWang:2004} estimated the Manning's coefficient
in the SWE within the variational framework using
the limited memory quasi-Newton method, and compared its performance
with several other optimization algorithms. However, these works have
considered only the situation of recovering a few parameters (with a
maximum three), instead of estimating a distributed Manning's
coefficient like here. If the number of unknowns is small, the
ill-posed nature of the problem does not evidence directly. Therefore,
the present work represents a nontrivial step towards the important
task of estimating distributed Manning's roughness coefficients.

\section{Linearization of the forward map}

In this section we describe the linearization of the forward map
$F:d_f\rightarrow u(d_f)$, where $u(d_f)$ denotes the solution to
system~\eqref{eqn:dsw}. The linearization is required for solving the
forward problem (with a predictor-corrector method) and the inverse
problem (adjoint and sensitivity problems, see Section 3). Therefore,
its derivation is of independent interest. In order to make the
presentation accessible, we choose to derive the derivative operator
informally. A rigorous derivation can be found in
Appendix~\ref{app:derivative}.

The bilinear form of problem~\eqref{eqn:dsw} is
\begin{align*}
  B(u,w)&=\int_\Omega u_tw dx + \int_\Omega k(u,\nabla u)\nabla u\cdot\nabla wdx\\
  & = \left(u_t,w\right) + \left(k(u,\nabla u)\nabla u,\nabla
    w\right),
\end{align*}
and the linear form is
\begin{equation*}
  \ell(w) =\int_\Omega fwdx + \int_{\Gamma_N}hwds = (f,w) + (h,w)_{\Gamma_N}.
\end{equation*}
The weak formulation of the problem reads: For almost all $t\in(0,T]$, find $u$ with the
given Dirichlet boundary condition and initial data $u(0)=u_0$ such that
\begin{equation*}
  B(u,w) = \ell(w)\quad \forall w\in V,
\end{equation*}
where $V$ is an appropriate function space~\cite{SantillanaDawson:2010}.

We shall seek the G\^{a}teaux derivative of the bilinear form $B$ at
$u$, that is, $\frac{d}{d\epsilon}B(u+\epsilon v,w)|_{\epsilon=0}$. We
aim at deriving an explicit formula to facilitate further
developments. We proceed as follows. It follows from the product rule
for differentiation that
\begin{align*}
  \left.\dfrac{\partial B(u+\epsilon
        v,w)}{\partial\epsilon}\right|_{\epsilon=0}
  =&\dfrac{\partial}{\partial\epsilon}\left.\left[\left(u_t + \epsilon
        v_t,w\right)+ \left(d_f\dfrac{\left[\left(u+\epsilon
              v\right)-z\right]^\alpha}{ \left|\nabla u+\epsilon\nabla
            v\right|^{1-\gamma}}\left(\nabla u +\epsilon\nabla
          v\right),\nabla w
      \right) \right]\right|_{\epsilon=0}\\
  =&\left(v_t,w\right) + \left(d_f\frac{(u-z)^\alpha} {\left|\nabla
        u\right|^{1-\gamma}}\nabla v,\nabla w\right) + I + II,
\end{align*}
where the terms $I$ and $II$ are respectively given by
\begin{align*}
  I & =\left.\left (d_f\,\dfrac{\partial [u+\epsilon
      v-z]^\alpha}{\partial\epsilon}\frac{\nabla u+ \epsilon\nabla v}{|\nabla u
      +\epsilon\nabla v|^{1-\gamma}},\nabla w\right)\right|_{\epsilon=0}
     = \left(d_f\,\alpha\,\dfrac{\left(u-z\right)^{\alpha-1}}
      {\left|\nabla u\right|^{1-\gamma}}\,v\nabla u,\nabla w\right),
\end{align*}
and
\begin{align*}
  II & =\left.\left(d_f[u+\epsilon v-z]^\alpha \frac{\partial |\nabla
        u+\epsilon\nabla v|^{\gamma-1}}{\partial\epsilon} \left(\nabla
        u+\epsilon\nabla v\right),
      \nabla w\right)\right|_{\epsilon=0}\\
  & = \left(d_f\,\left(u-z\right)^\alpha \,(\gamma-1)\left|\nabla
      u\right|^{\gamma-2}
    \dfrac{\nabla u}{|\nabla u|}\cdot\nabla v \,\nabla u,\nabla w\right)\\
  & = \left(d_f\,(\gamma-1)\,\dfrac{\left(u-z\right)^\alpha}{\left|\nabla
      u\right|^{3-\gamma}}
    \nabla u\cdot\nabla v \,\nabla u,\nabla w\right).
\end{align*}
Here the second line follows from the relation $|\nabla u| =
\sqrt{\nabla u\cdot\nabla u} = (\nabla u\cdot\nabla u)^\frac{1}{2}$
that implies
\begin{align*}
  \left.\dfrac{\partial|\nabla u+\epsilon\nabla
      v|}{\partial\epsilon}\right|_{\epsilon=0} &=\dfrac{1}{2}\left(\left(\nabla
  u+\epsilon\nabla v\right)\cdot\left(\nabla u+\epsilon\nabla
  v\right)\right)^{-\frac{1}{2}}\,
  2(\nabla u+\epsilon\nabla v)\cdot\nabla v|_{\epsilon=0}
  =\dfrac{\nabla u\cdot\nabla v}{|\nabla u|}.
\end{align*}

Consequently, by combining all these identities, we arrive at the
following formula
\begin{align*}
  \left. \dfrac{\partial B\left(u+\epsilon
        v,w\right)}{\partial\epsilon} \right|_{\epsilon=0}
  =&\left(v_t,w\right) + \left(d_f\dfrac{(u-z)^\alpha}{|\nabla
      u|^{1-\gamma}}\nabla
    v,\nabla w\right)\\
  & \quad+ \left(d_f\alpha\frac{(u-z)^{\alpha-1}}{|\nabla u|^{1-\gamma}}
    \,v\,\nabla u,\nabla w\right)\\
  &\quad+ \left(d_f(\gamma-1)\frac{(u-z)^\alpha}{|\nabla u|^{1-\gamma}}
    \,\dfrac{\nabla u}{|\nabla u|}\cdot\nabla v \,
    \dfrac{\nabla u}{|\nabla u|},\nabla w\right)\\
  = &\left(v_t,w\right) + \left(k(u,\nabla
      u)\left(I-\left(1-\gamma\right)\tilde{\eta}
      \otimes\tilde{\eta}\right)\cdot\nabla v,\nabla w\right) \\
   &\quad+ \left(k(u,\nabla u)\,\frac{\alpha}{(u-z)}\, v\,\nabla
    u,\nabla w\right)
\end{align*}
where $I$ is the identity operator and the vector field
$\tilde{\eta}=\tfrac{\nabla u}{|\nabla u|}$ is the normalized gradient
vector field. The matrix-valued function $\tilde{\eta}\otimes\tilde{\eta}$ represents a
projection operator onto the gradient direction $\tilde{\eta}$. Hence,
the structure of the second term indicates that, for the linearized
problem, the diffusion along the gradient direction is attenuated by
$1-\gamma$, whereas the tangential component is not affected.  To
simplify notation we denote this attenuated diffusion tensor as
\[
  k_{\eta\eta}(u,\nabla u)=k(u,\nabla u)\left(I-\left(1-\gamma\right)\tilde{\eta} \otimes\tilde{\eta}\right).
\]
Meanwhile, the linearized problem has a convection term (the third
term), as a consequence of the nonlinear term involving $u$.  These
structural terms relate to the underlying physics of the model.

It follows directly from the definition of the G\^{a}teaux derivative, i.e., which is denoted by 
$v=u'(d_f)d\in V$ and characterizes the perturbation of $u(d_f)$ caused by a small
perturbation of the coefficient $d_f$ in the direction $d$ that it (in weak
formulation) satisfies
\begin{align*}
  \left(v_t,w\right) &+ \left(k_{\eta\eta}(u,\nabla u)
    \cdot\nabla v,\nabla w \right) + \left(\dfrac{\alpha\, k(u,\nabla
        u)}{(u-z)}\, v\,\nabla u,\nabla w\right) = -\left( d
    \dfrac{(u-z)^\alpha}{\left|\nabla u\right|^{1-\gamma}}\nabla
    u,\nabla w \right)
\end{align*}
and the initial condition is $v(0)=0$, since the initial data is not
affected by a perturbation of the friction coefficient.

\section{Inversion algorithm}

Now we turn to the inverse problem of reconstructing the coefficient
$d_f$ from the measurements of water heights. As a general rule, the
inverse problem is ill-posed in the sense that small perturbations in
the data can lead to large changes in the solution.  Hence we adopt a
regularization strategy by incorporating a penalty term into the cost
functional, following the pioneering idea of
Tikhonov~\cite{TikhonovArsenin:1977}.  More precisely, we consider the
following penalized misfit functional
\begin{equation*}
  J(d_f) = \frac{1}{2}\int_0^T\int_\Omega(u(d_f)-g)^2 dxdt +\frac{\delta}{2}\int_\Omega |\nabla d_f|^2 dx,
\end{equation*}
where the scalar $\delta$ is the regularization parameter, and $g$
denotes the noisy measurements of the water height $u(d_f)$. With
minor modifications, the algorithm discussed below can also be applied
to other measurements, for example, water height on the boundary or
scattered in the domain. The term $\|\nabla d_f\|_{L^2(\Omega)}^2$
enforces smoothness on the sought-for coefficient, and thereby
restores the numerical stability necessary for practical
computations. To numerically minimize the functional, we adopt the
conjugate gradient method. The method is of gradient descent type, and
it only requires evaluating the gradient of the functional $J(d_f)$ at
each step. We note that the conjugate gradient method has been
successfully applied to a wide variety of practical inverse problems,
such as in heat transfer and mechanics; see for
example,~\cite{Alifanov:1994,JinZou:2010} and references therein for
details.

To derive a computationally efficient gradient formula, we first note
that, given a (descent) direction $d$, the misfit term in the
functional $J$ can be approximated using a Taylor expansion and ignoring
higher order terms.
\begin{equation*}
  \begin{aligned}
    \frac{1}{2}\int_0^T\int_\Omega
    \left(u(d_f+d)-g\right)^2 &dxdt
    - \frac{1}{2}\int_0^T\int_\Omega\left(u(d_f\right)-g)^2dxdt=\\
    =&\frac{1}{2}\int_0^T\int_\Omega \left(u(d_f
        +d)-u(d_f)\right)\left(u(d_f
        +d)-g+u(d_f)-g\right)dxdt\\
    \approx&\int_0^T\int_\Omega u'(d_f)d
    \,\left(u(d_f)-g\right)dxdt.
  \end{aligned}
\end{equation*}
The approximation is reasonable if the magnitude of the direction $d$ is small.

The last formula can be further simplified with the help of the
adjoint problem for $p$, which in weak form reads
\begin{align*}
  \left(-p_t,w\right)+\left(k_{\eta\eta}(u,\nabla u)\cdot\nabla
    p,\nabla w\right) +\left(\dfrac{\alpha\, k(u,\nabla
        u)}{(u-z)}\, \nabla u\cdot\nabla
    p,w\right)=\left(u(d_f)-g,w\right)
 \end{align*}
together with the terminal condition $p(T)=0$. Recall the weak
formulation of the sensitivity problem $v=u'(d_f)d$, that is,
\begin{align*}
  \left(v_t,w\right) + \left(k_{\eta\eta}(u,\nabla u)\cdot
  \nabla v,\nabla w\right) +
  \left(\dfrac{\alpha\, k(u,\nabla u)}{(u-z)}\, v\,\nabla
  u,\nabla w\right)= - \left(d\frac{(u-z)^\alpha}{|\nabla u|^{1-\gamma}}\nabla
  u,\nabla w\right),
\end{align*}
together with the initial condition $v(0)=0$. Upon setting the test
function $w=u'(d_f)\,d$ and $w=p$ in the weak formulations for $p$ and
$u'(d_f)\,d$, respectively, we arrive at
\begin{align*}
  \int_0^T\left(u(d_f)-g,u'(d_f)\,d\right)dt
  = -\int_0^T\int_\Omega& d\,\frac{(u-z)^\alpha}{|\nabla
    u|^{1-\gamma}}\nabla p\cdot
  \nabla udxdt-\int_0^T\frac{d}{dt}(p,u'(d_f)d)dt\\
  =-\int_0^T\int_\Omega &d\,\frac{(u-z)^\alpha}{|\nabla
    u|^{1-\gamma}}\nabla p\cdot\nabla udxdt,
\end{align*}
where the last identity follows from the initial condition for
$u'(d_f)\,d$ and terminal condition for $p$. This relation yields the
following concise gradient formula of the functional $J(d_f)$
\begin{equation*}
  J'(d_f) = -\int_0^T\frac{(u-z)^\alpha}{|\nabla u|^{1-\gamma}}
\nabla p\cdot\nabla udt -\delta\Delta d_f.
\end{equation*}
We note that this gradient $J'(d_f)$ is inappropriate for updating the
coefficient $d_f$ directly due to its lack of desired regularity. The
consistent gradient of the functional with respect to $H^1(\Omega)$,
denoted by $J_s'(d_f)$, can be calculated as
\begin{equation*}
  -\Delta J_s'(d_f) + J_s'(d_f) = J'(d_f)
\end{equation*}
with a homogeneous Neumann boundary condition.

Now we can give a complete description of the conjugate gradient
method summarized in Algorithm~\ref{alg:cgm}. In the algorithm, one
has the freedom to choose the conjugate coefficient $\beta_k$ and the
step size $\theta_k$. There are several viable choices of the
conjugate coefficient~\cite{HagerZhang:2006}. One popular choice is
suggested by Fletcher-Reeves, which reads
\begin{equation*}
  \beta_{k-1} = \frac{\|J_s'(d_f^k)\|_{L^2(\Omega)}^2}{\|J_s'(d_f^{k-1})\|_{L^2(\Omega)}^2}
\end{equation*}
with the convention $\beta_0=0$, and then update the conjugate
direction $d_k$ with
\begin{equation*}
d_k = J_s'(d_f^k)  + \beta_{k-1}d_{k-1}.
\end{equation*}
Generally, the step size selection is of crucial importance for the
performance of the algorithm. We have opted for the following simple
rule. By means of a Taylor expansion of the objective function
$J(d_f^k-\theta d_k)$, with the forward solution $u(d_f^k-\theta d_k)$
linearized around $d_f^k$, we arrive at the following approximate
formula for determining an appropriate step size $\theta_k$
\begin{equation*}
  \theta_k = \frac{\langle r_k,u'(d_f^k)d_k\rangle_{L^2(0,T;L^2(\Omega))}+\delta\langle\nabla d_f^k,\nabla d_k\rangle_{L^2(\Omega)}}{
    \|u'(d_f^k)d_k\|_{L^2(0,T;L^2(\Omega))}^2+\delta\|\nabla d_k\|_{L^2(\Omega)}^2},
\end{equation*}
where $r_k=u(d_f^k)-g$ denotes the misfit (residual).  The
step size $\theta_k$ is determined to enforce a reduction the
functional value, that is,
$J\left(d_f^k-\theta_k\,J_s'(d_f^k)\right)\leq
J(d_f^k)$. Our experience with other inverse problems
indicates that this choice works reasonably well in
practice~\cite{JinZou:2010}. Advanced step size selection rules, such
as, Barzilai-Borwein rule with backtracking, maybe also be adopted to
further enhance the performance. The algorithm terminates if the
selected step size falls below $1.0\times10^{-3}$. Overall, each step
of the iteration invokes three forward solves: the (nonlinear) forward
solve for computing the map $u(d_f)$, the (linear) adjoint
solve for calculating the adjoint $p(d_f)$ and consequently
the gradient $J'(d_f)$ and the (linear) sensitivity solve
for selecting the step size $\theta$. The extra computational effort
for computing the smoothed gradient $J_s'(d_f)$ is marginal
compared with other steps due to its simple structure.

\begin{algorithm}[t]
	\caption{Conjugate gradient method.\label{alg:cgm}}
\begin{algorithmic}[1]
  \STATE Set $k=0$ and choose initial guess $d_f^0$.
  \REPEAT
  \STATE Solve direct problem with $d_f=d_f^k$, and determine
  residual $r_k=u(d_f^k)-g$.
  \STATE Solve adjoint problem with right hand side $r_k$.
  \STATE Calculate gradient $J_s'(d_f^k)$,
  conjugate coefficient $\beta_k$, and direction $d_k$.
  \STATE Solve the sensitivity problem with direction $d=d_k$.
  \STATE Compute step length $\theta_k$ in conjugate
  direction $d_k$.
  \STATE Update coefficient $d_f^k=d_f^k-\theta_kd_k$.
  \STATE Increase $k$ by one.
  \UNTIL A stopping criterion is satisfied.
  \STATE Output approximation $d_f$
\end{algorithmic}
\end{algorithm}

\section{Numerical experiments and discussions}
Here we present some numerical results for one-dimensional examples
to illustrate the feasibility of the proposed inversion
technique. The forward problem is discretized using piecewise linear
finite elements in space and the generalized-$\alpha$ method in time
(detailed in Appendix~\ref{app:gam}). The adjoint and sensitivity
problems are both solved with the generalized-$\alpha$ method.

The spatial domain $\Omega=[-2,2]$, and the mesh size $h$ is
$\frac{1}{4}$. The time interval is $\left[0,\frac{1}{2}\right]$, and
the time step size is $\frac{1}{40}$. This mesh was used for both
generating the exact data and used in the inversion step (i.e.,
adjoint problem and sensitivity problem). We note that we also
experimented with using finer mesh for generating the exact data, and
the reconstructions are identical. Also both the forward solution $u(d_f)$
and the coefficient $d_f$ are represented in this mesh. The initial
guess for the coefficient is $d_f=1$. The noisy data $g$ are generated
pointwise as
\begin{equation*}
  g = u(d_f^\dagger) + \varepsilon\max_{(x,t)\in\Omega\times[0,T]}\left(\left|u(d_f^\dagger)\right|\right)\zeta,
\end{equation*}
where $\varepsilon$ is the relative noise level, and the random
variable (noise) $\zeta$ follows a standard Gaussian
distribution. The choice of the regularization parameter $\delta$ is
crucial in any regularization
strategies~\cite{TikhonovArsenin:1977}. There have been intensive
studies on its appropriate choice which have led to systematical and
rigorous rules for choosing an appropriate value,
see~\cite{Jin:2009,ItoJinTakeuchi:2011} for recent progress. However,
in this preliminary study, we have opted for the conventional
trial-and-error approach.

We consider three examples: one with a smooth coefficient, and two
with a discontinuous coefficient. First, we consider the recovery of a
continuous coefficient.
\begin{exam}\label{exam:cont}
  The forward problem has a homogeneous Neumann boundary condition,
  and the initial condition $u_0$ is
  $u_0=-\frac{1}{4}x+\frac{3}{2}$. The exact coefficient is
  $d_f^\dagger(x)=1+\frac{1}{16}(x^2-4)^2$.
\end{exam}

Figure~\ref{fig:exam1conv}(a) and Table~\ref{tab:err} show the
numerical results for Example~\ref{exam:cont}, where $e$ is the
relative error of an approximation $d_f$, defined as
$e=\|d_f-d_f^\dagger\|_{L^2(\Omega)}/\|d_f^\dagger\|_{L^2(\Omega)}$. The
reconstructions are in reasonable agreement with the exact coefficient 
$d_f^\dagger$ for up to
$2\%$ noise in the data. Hence the proposed method is stable and
accurate. We note that the approximation near the boundary seems less
accurate compared to other regions. The error $e$ decreases as the
noise level $\epsilon$ decreases to zero, see also
Table~\ref{tab:err}. Overall, the convergence of the inversion
algorithm is rather steady, see Figures~\ref{fig:exam1conv}(b) and
(c).  While the functional value $J(d_f^k)$ decreases monotonically as
the iteration proceeds, the convergence of the error $e$ exhibits a
clear valley, indicating that a premature termination of the algorithm
might result in sub-optimal reconstructions.

\begin{table}
  \centering
  \caption{Numerical results (error $e$) for different noise levels.}
  \begin{tabular}{ccccc}
  \hline
  $\varepsilon$ & $0\%$ & $0.5\%$ & $1\%$ & $2\%$\\
  \hline
  Example~\ref{exam:cont}    & 5.94e-3 & 2.00e-2 & 2.90e-2 & 4.52e-2\\
  Example~\ref{exam:discont} & 4.49e-2 & 6.41e-2 & 7.49e-2 & 9.99e-2\\
  Example~\ref{exam:disc2}   & 4.05e-2 & 5.15e-2 & 5.94e-2 & 9.42e-2\\
  \hline
  \end{tabular}\label{tab:err}
\end{table}

\begin{figure}
\centering
  \begin{tabular}{ccc}
    \includegraphics[width=5.5cm]{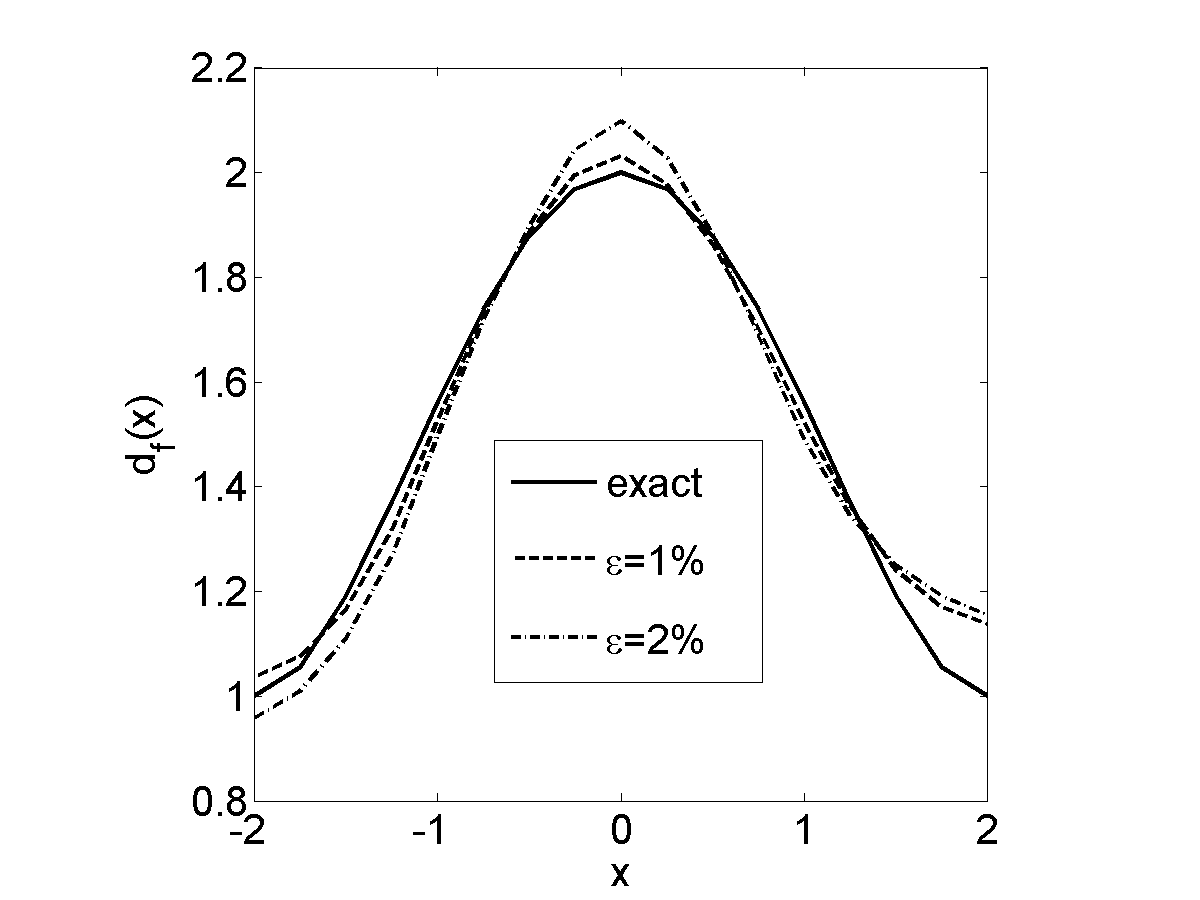} & \includegraphics[width=5.5cm]{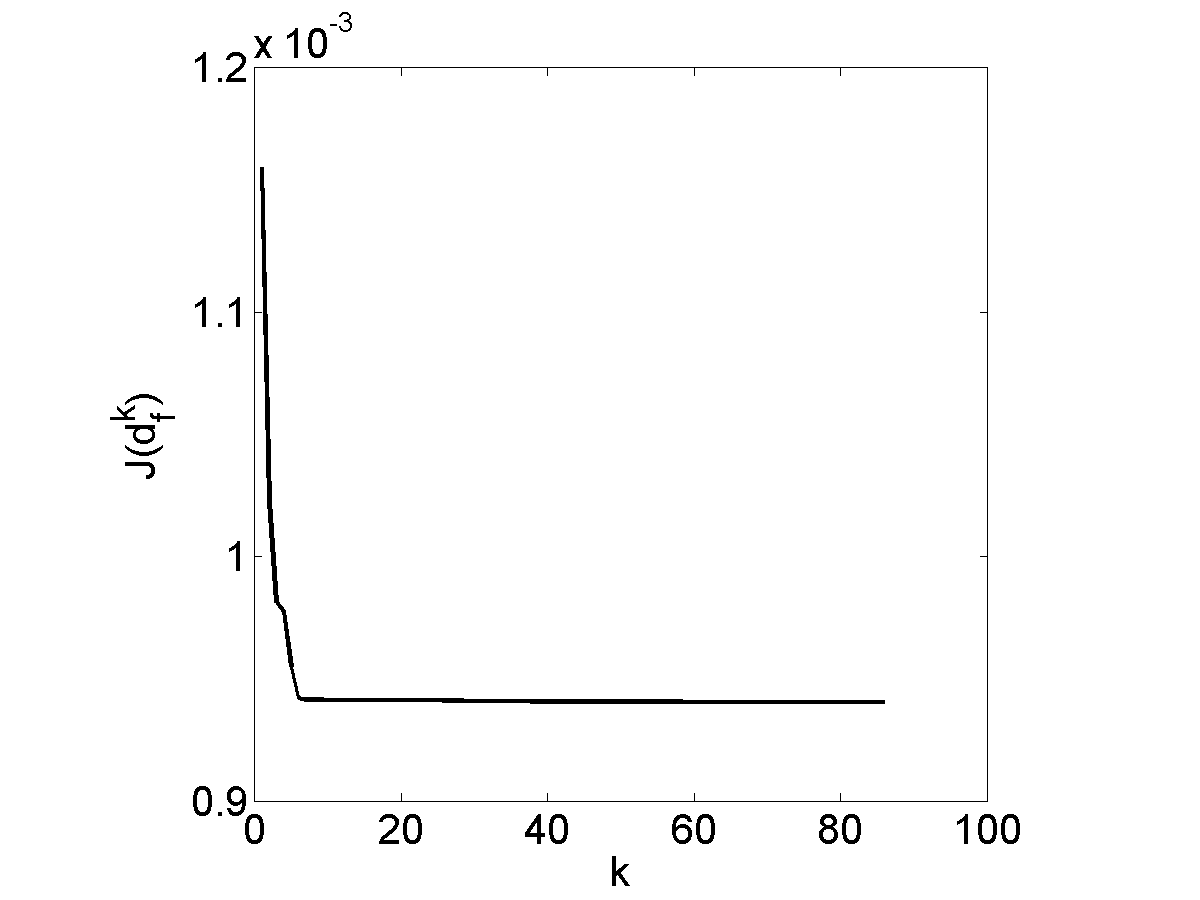}
    & \includegraphics[width=5.5cm]{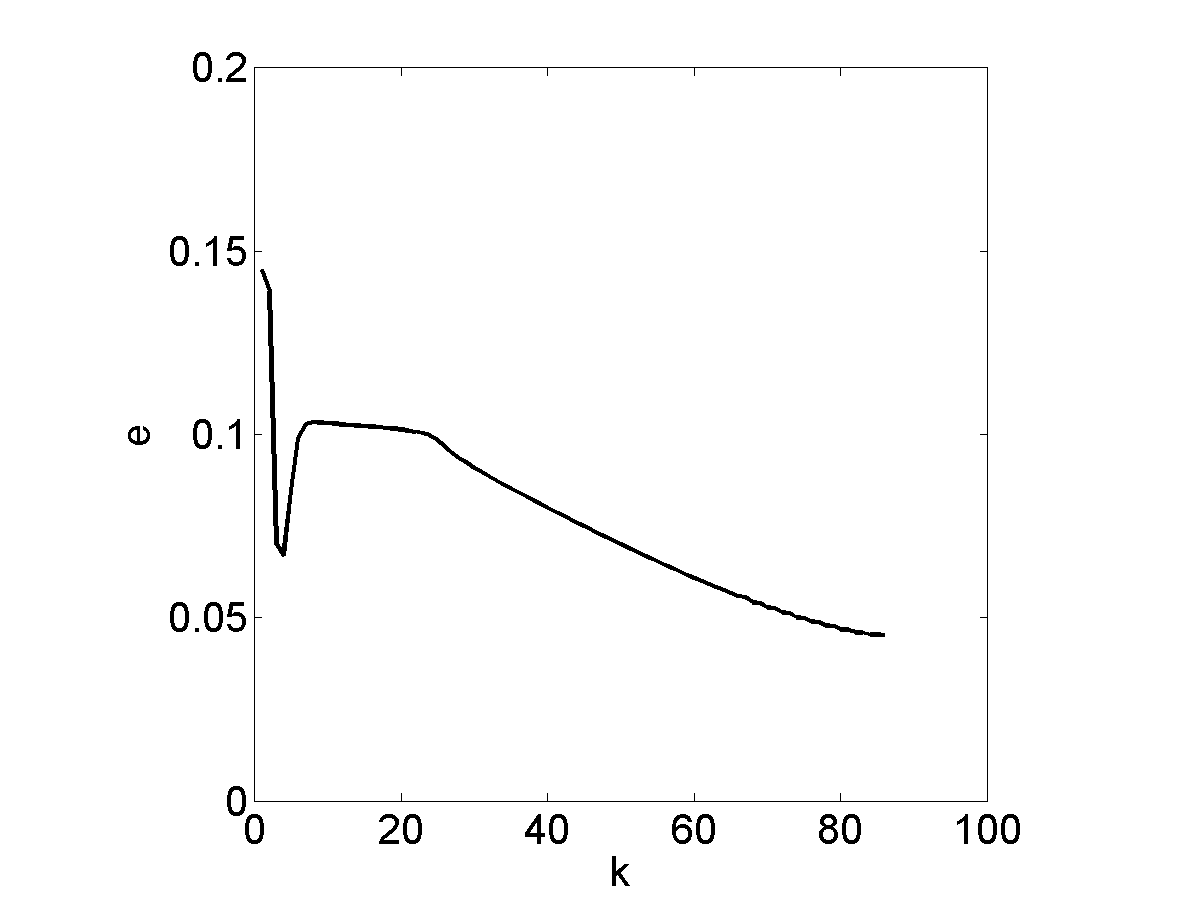}\\
     (a) reconstructions & (b) functional value $J(d_f^k) $ & (c) error $e$
  \end{tabular}
  \caption{Numerical results for Example~\ref{exam:cont}. Here the
    convergence of Algorithm~\ref{alg:cgm} is for $\varepsilon=2\%$
    noise.}
  \label{fig:exam1conv}
\end{figure}

Then we consider the recovery of a discontinuous coefficient.
\begin{exam}\label{exam:discont}
  The boundary condition and initial condition of the problem are
  identical to those in Example~\ref{exam:cont}. The exact
  coefficient is $d_f^\dagger=1+\chi_{[-\frac{5}{4},\frac{3}{4}]}$,
  where $\chi$ denotes the characteristic function.
\end{exam}

Figure~\ref{fig:exam2conv}(a) and Table~\ref{tab:err} present the
numerical results for Example~\ref{exam:discont}. The convergence of
the result with respect to the noise level $\varepsilon$ is again
clearly observed. The reconstructions capture the overall shape of the
true solution. However, the discontinuities are not well resolved,
even for exact data, and consequently the results are less accurate
compared with those for Example~\ref{exam:cont}. This is attributed to
the presence of discontinuities in the sought-for solution
$d_f^\dagger$, which cannot be accurately approximated using the
smoothness penalty $|\nabla d_f|_{L^2(\Omega)}^2$. Discontinuity
preserving penalties, such as, total variation, might be employed to
improve the resolution.  Nonetheless, the conjugate gradient algorithm
remains fairly steady, see Figures~\ref{fig:exam2conv}(b) and (c).

\begin{figure}
  \centering
  \begin{tabular}{ccc}
    \includegraphics[width=5.5cm]{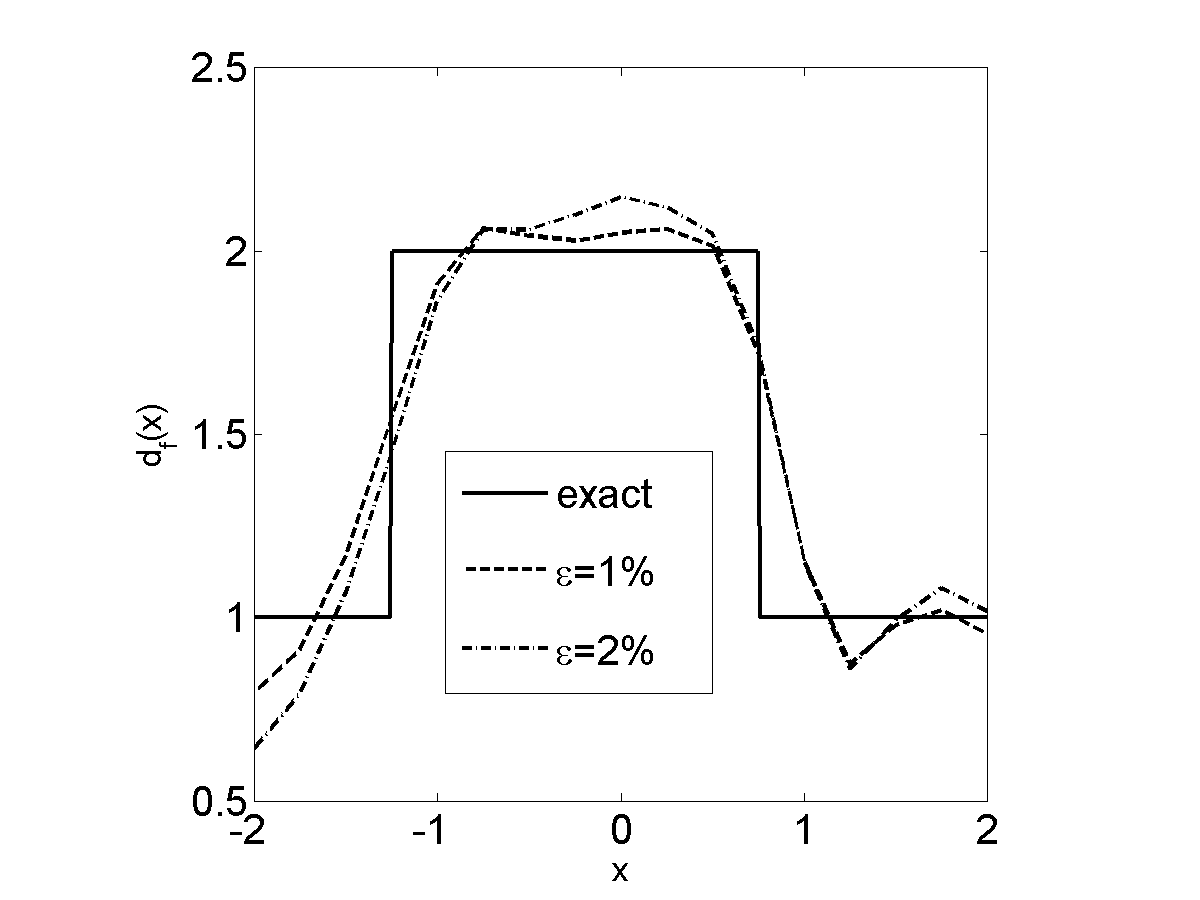} & \includegraphics[width=5.5cm]{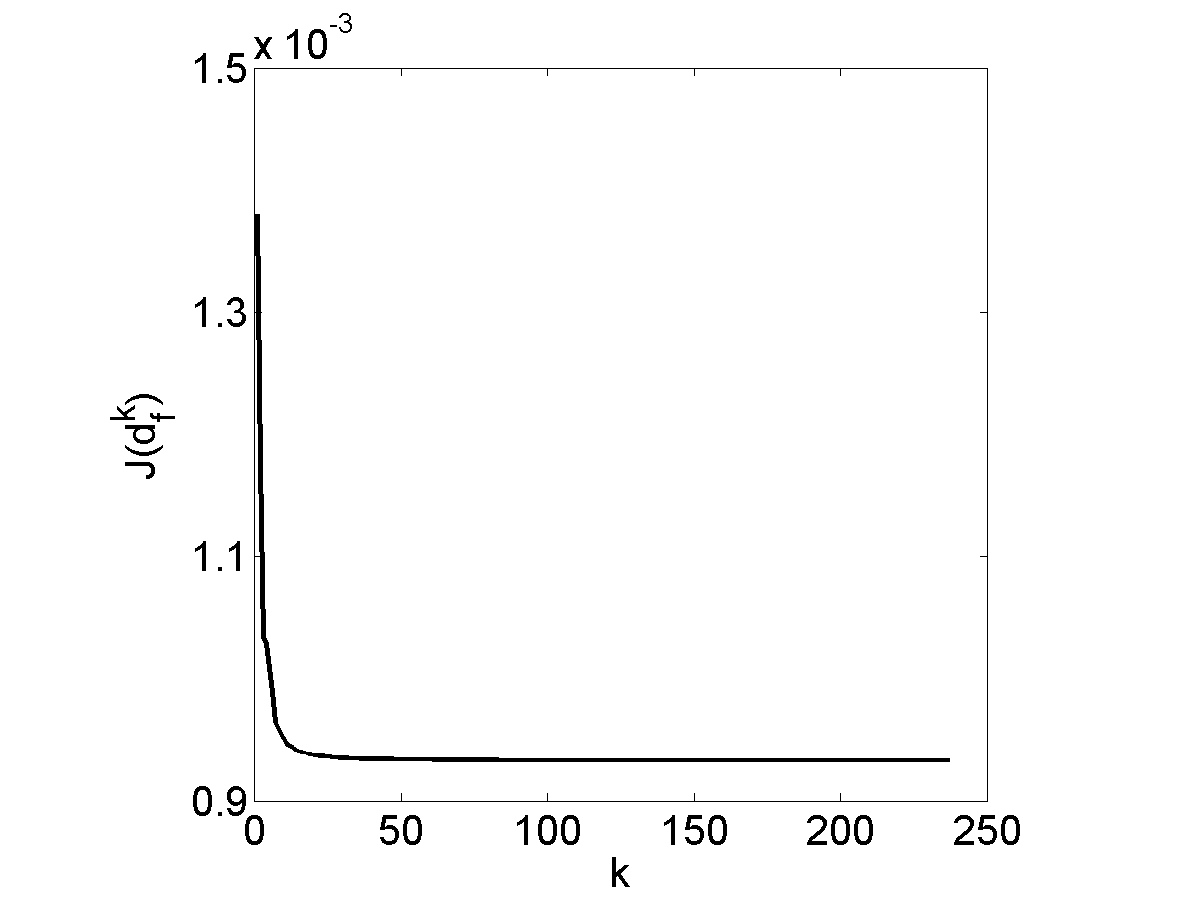}
     & \includegraphics[width=5.5cm]{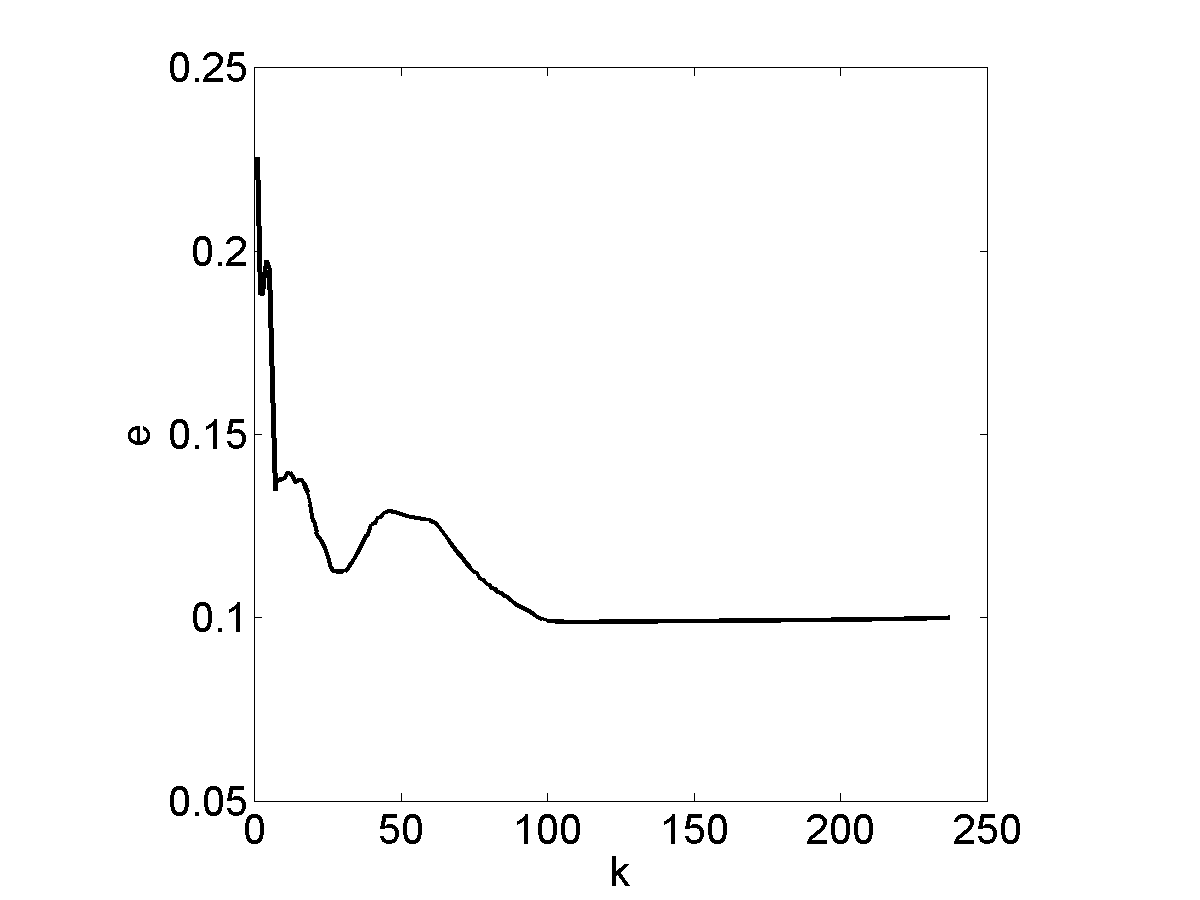}\\
    (a) reconstructions & (b) functional value $J(d_f^k)$ & (c) error $e$
  \end{tabular}
  \caption{Numerical results for Example~\ref{exam:discont}. Here the
    convergence of Algorithm~\ref{alg:cgm} is for $\varepsilon=2\%$
    noise.}
  \label{fig:exam2conv}
\end{figure}

A last example considers the recovery of a more complex coefficient profile.
\begin{exam}\label{exam:disc2}
  The boundary condition and initial condition of the problem are
  identical with those in Example~\ref{exam:cont}. The exact
  coefficient $d_f^\dagger(x)$ is given by
  $d_f^\dagger=1-\frac{1}{2}\chi_{[-\frac{7}{8},-\frac{3}{8}]}+\frac{1}{2}\chi_{[\frac{5}{8},\frac{9}{8}]}$.
\end{exam}

Here the true solution has more refined details, and hence the spatial
mesh size $h$ is accordingly refined to $\frac{1}{8}$ for a better
resolution. The results for Example~\ref{exam:disc2} are shown in
Figure~\ref{fig:exam3conv} and Table~\ref{tab:err}. The convergence of
the numerical reconstruction with respect to the noise level is again
observed, see Table~\ref{tab:err}. The observations for the previous
example remain valid: the numerical reconstructions roughly capture
the profile of the true solution, but fail to resolve accurately the
discontinuities, and the algorithm converges steadily and reasonably
quick.

\begin{figure}
  \centering
  \begin{tabular}{ccc}
    \includegraphics[width=5.5cm]{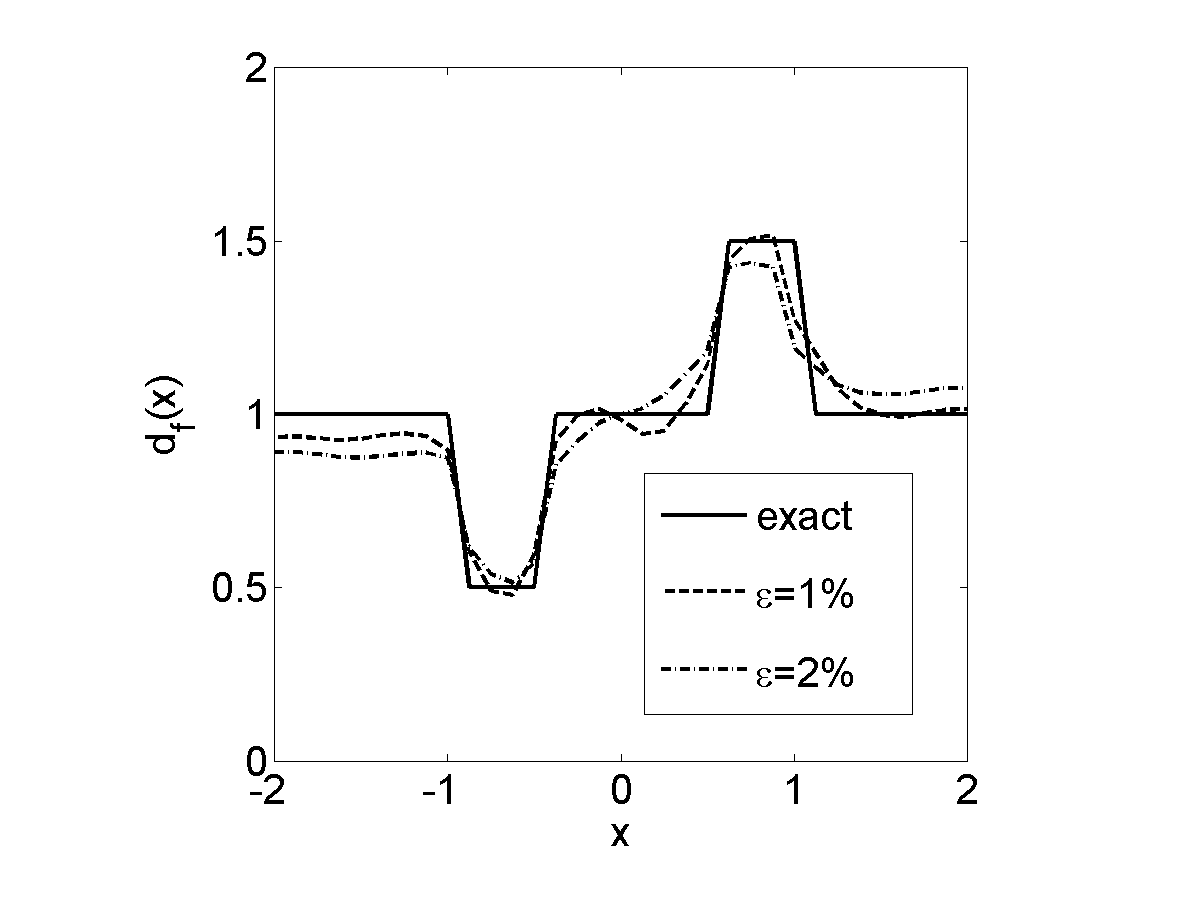} & \includegraphics[width=5.5cm]{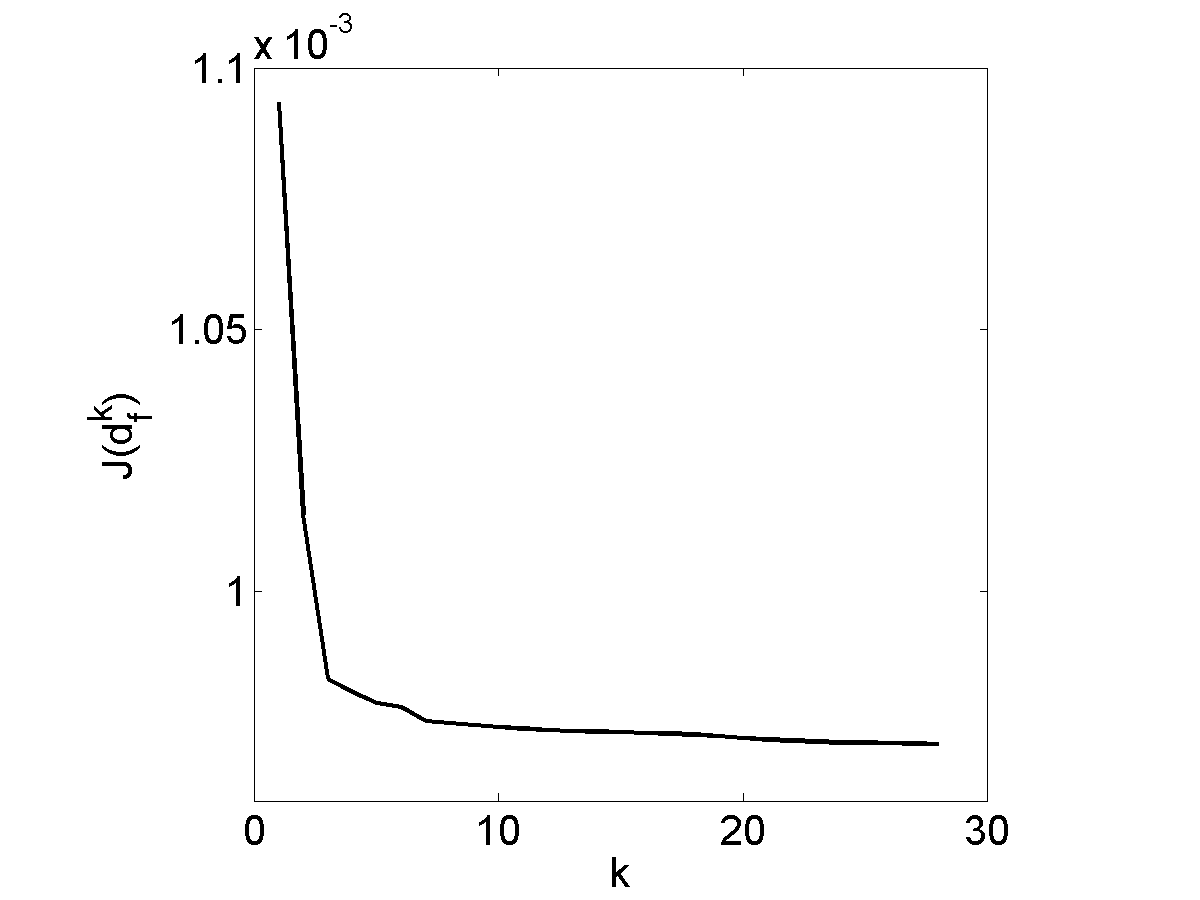}
     & \includegraphics[width=5.5cm]{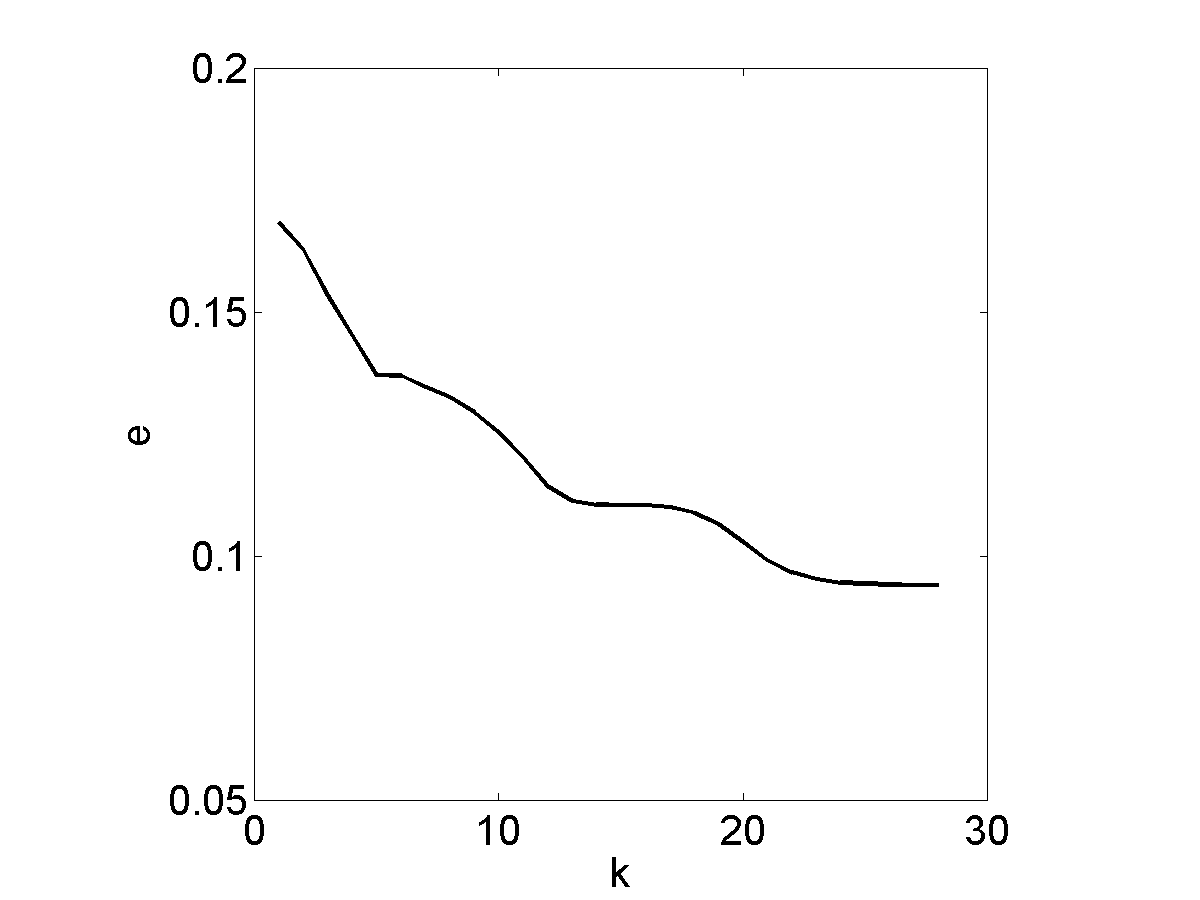}\\
    (a) reconstructions & (b) functional value $J(d_f^k)$ & (c) error $e$
  \end{tabular}
  \caption{Numerical results for Example~\ref{exam:disc2}. Here the
    convergence of Algorithm~\ref{alg:cgm} is for $\varepsilon=2\%$
    noise.}
  \label{fig:exam3conv}
\end{figure}

\section{Concluding remarks}
We have presented an inversion technique for estimating the Manning's
coefficient in the diffusive wave approximation of the shallow water
equations. The results show that the proposed approach is capable of yielding
an accurate and stable estimate in the presence of noise.  We have also detailed a
careful study of the properties of the forward map, in particular, we
discuss its continuity and differentiability based on maximal
regularity theory for parabolic problems. The mathematical analysis,
such as, convergence and convergence rates, of such an inversion
technique remains to be investigated.  Also the evaluation of the
method on real data is of significant interest.

\section*{Acknowledgements}
This work was initiated while V.M.C. was a Visiting Professor at the
Institute for Applied Mathematics and Computational Science (IAMCS),
Texas A\&M University, College Station. The work of M.G. was carried
out during his visit at IAMCS. They would like to thank the institute
for the kind hospitality and support. The work of B.J. is supported by
Award No. KUS-C1-016-04, made by King Abdullah University of Science
and Technology (KAUST).

\appendix

\section{Properties of the forward map}\label{app:derivative}
In this part, we briefly discuss the continuity and differentiability
of the forward map $F: L^\infty\rightarrow L^2(0,T;H^1(\Omega)),\ d_f
\mapsto u(d_f)$ based on maximal regularity theory for parabolic
problems~\cite{Groger:1992}. The conditions in Theorem~\ref{thm:cont}
impose a certain regularity constraint on the coefficient $d_f$ as
well as on the boundary and initial conditions. Such mapping
properties are essential for analyzing commonly used regularization
schemes, for example, Tikhonov regularization and Landweber iteration
for solving the inverse problem, and for establishing the convergence
of numerical algorithms.

We first show the Lipschitz continuity of the forward map $F$.
\begin{thm}\label{thm:cont}
  Assume that $\|u(d_f)\|_{L^\infty}$ and $\|\nabla
  u(d_f)\|_{L^\infty}$ are uniformly bounded, and that $d_f$, $|\nabla
  u(d_f)|$ and $(u(d_f)-z)$ are strictly positive, and further the
  gradient $|\nabla (t\,u(d_f)+(1-t)u(\tilde{d}_f))|$ is strictly
  positive for all $t\in(0,1)$ and $d_f,\tilde{d}_f$ in the admissible
  set $\mathcal{A}$. Then if $\gamma$ is sufficiently close to unity,
  the mapping $F:L^\infty \to L^2(0,T;H^1(\Omega))$ given by $d_f
  \mapsto u(d_f)$ is Lipschitz continuous on $\mathcal{A}$.
\end{thm}
\begin{proof}
  We denote by $k\left(u,\nabla
    u;d_f\right)=d_f\frac{(u-z)^\alpha}{|\nabla u|^{1-\gamma}}$ and
  $\tilde{u}=u(\tilde{d}_f)$, and let $v=u-\tilde{u}$. We denote the
  bilinear form parametrized by $d_f$ as
\begin{align*}
  B(u,w;d_f)= (u_t,w) + \left(k(u,\nabla u;d_f)\nabla u,\nabla
    w\right),
\end{align*}
By subtracting the bilinear forms $B(u,w;d_f) =
  \ell(w)$ and $B(\tilde{u},w,\tilde{d_f}) = \ell(w)$ and choosing
  $w=v$, we arrive at
\begin{equation*}
  0 = \left(u_t -\tilde{u}_t,w\right) + \left(k(u,\nabla u;d_f)\nabla
  u-k(\tilde{u},\nabla \tilde{u};\tilde{d_f})\nabla \tilde{u},\nabla w\right),
\end{equation*}
which by virtue of the assumptions on $u$ and $\nabla u$ can be
rearranged into
\begin{equation*}
  \begin{aligned}
    \dfrac{1}{2}\partial_t\|v\|^2_{L^2} + C_K\|\nabla v\|^2_{L^2}
    \leq& -\left(k(u,\nabla u;{d_f-\tilde{d_f}})\nabla
    u,\nabla v\right)\\ &-\left(\left(k(u,\nabla
    u;{\tilde{d_f}})-k(\tilde{u},\nabla
    \tilde{u};{\tilde{d_f}})\right)\nabla \tilde{u},\nabla
    v\right):=I+II,
\end{aligned}
\end{equation*}
where $C_K$ is the coercivity constant for the bilinear form $B(\cdot,\cdot)$.  
Using Cauchy-Schwarz inequality and Young's inequality, the first
summand $I$ on the right hand side can be estimated as follows
\begin{equation*}
  I\le C(\epsilon_1) \| d_f -\tilde{d_f}\|^2_{L^\infty} + \epsilon_1\|\nabla v\|^2_{L^2}.
\end{equation*}
Meanwhile, we split the nonlinear term in the bracket in the second summand $II$ into
\begin{equation}\label{eqn:kdecom}
  k(u,\nabla u;\tilde{d_f})
  -k(\tilde{u},\nabla \tilde{u};\tilde{d_f})
  =d_f\left[\frac{ \left(u-z\right)^\alpha\left( |\nabla{\tilde{u}}|^{1-\gamma} - |\nabla{u}|^{1-\gamma}\right)}
    {|\nabla{u}|^{1-\gamma}|\nabla{\tilde{u}}|^{1-\gamma}}
    + \frac{\left(u-z\right)^\alpha - \left(\tilde{u}-z\right)^\alpha}{|\nabla{\tilde{u}}|^{1-\gamma}}\right].
\end{equation}
Now the mean value theorem gives
\begin{equation}\label{eqn:taylorz1}
  \left(u-z\right)^\alpha - \left(\tilde{u}-z\right)^\alpha = \alpha \left(\bar{u}-z\right)^{\alpha - 1}v,
\end{equation}
where $\bar{u}$ is an element between $u$ and $\tilde{u}$, and also by
means of the Taylor expansion
\begin{equation}\label{eqn:taylorgrad1}
  |\nabla{\tilde{u}}|^{1-\gamma}-|\nabla{u}|^{1-\gamma} = \left(1-\gamma\right)\mathbf{v}\cdot\nabla v,
\end{equation}
and the function \[\mathbf{v}=\int_0^1\frac{\nabla(u-sv)}{
  |\nabla(u-sv)|^{1+\gamma}}ds,\] which by assumption is bounded in
$L^\infty$. Consequently by Young's inequality, we get
\begin{equation*}
  \begin{aligned}
    II \leq& \left(1-\gamma\right) \|k(u,\nabla u,\tilde{d_f})
    \|_{L^\infty}\|\mathbf{v}\|_{L^\infty}
    \|\nabla{\tilde{u}}\|^\gamma_{L^\infty}\| \nabla v\|^2_{L^2} + C
    \left(\epsilon_2^{-1}\|v\|^2_{L^2} + \tfrac{\epsilon_2}{4}\|\nabla v\|^2_{L^2} \right)\\
    \leq &C \left(\left(1-\gamma\right)+\epsilon_2\right) \|\nabla v\|^2_{L^2} +
    C\epsilon_2^{-1} \|v\|^2_{L^2}.
 \end{aligned}
\end{equation*}
Since $\gamma$ is close to unity and for sufficiently small
$\epsilon_1$, $\epsilon_2$, $\mu:=C\left(\left(1-\gamma\right)+\epsilon_1 +
\epsilon_2\right)< C_K$, we obtain
\begin{equation*}
  \begin{aligned}
    \tfrac{1}{2}\partial_t\|v\|^2_{L^2} + \left(C_K-\mu\right) \|\nabla
    v\|^2_{L^2} \le C\epsilon_2^{-1} \| v\|^2_{L^2} + C \| d_f
    -\tilde{d_f}\|^2_{L^\infty}
\end{aligned}
\end{equation*}
Now an application of Gr\"ownwall's inequality leads to
\begin{equation*}
  \|v \|^2_{L^2} + \int_0^T \|\nabla v\|^2_{L^2}ds \le C \| d_f -\tilde{d_f}\|^2_{L^\infty}
\end{equation*}
upon noting the condition $u(0) = \tilde{u}(0)$.
\end{proof}

Our next result improves the regularity of the map in
Theorem~\ref{thm:cont} by invoking Gr\"{o}ger's maximal regularity
theory~\cite{Groger:1992}, which is needed for the differentiability.
\begin{thm}\label{thm:contWp}
  Let the assumptions in Theorem~\ref{thm:cont} be fulfilled. Then the
  mapping $F:L^\infty \to L^2(0,T;W^{1,p}(\Omega))$, $d_f \mapsto
  u(d_f)$ is Lipschitz continuous for some $p\in(2,\infty)$.
\end{thm}
\begin{proof}
  As before, we denote by \[k(u,\nabla
  u;d_f)=d_f\frac{\left(u-z\right)^\alpha}{|\nabla u|^{1-\gamma}}\] and
  $\tilde{u}=u(\tilde{d}_f)$, and let $v=u-\tilde{u}$. Then $v$ solves
\begin{equation*}
  v_t + Av = f
\end{equation*}
with
\begin{equation*}
  Av = -\nabla \cdot \left(\left(k(u,\nabla u;\tilde{d_f})-k(\tilde{u},\nabla \tilde{u};\tilde{d_f})\right)\nabla \tilde{u}
    + k(u,\nabla u;\tilde{d}_f)\nabla v\right)
\end{equation*}
and $f = \nabla \cdot \left(k(u,\nabla u;d_f-\tilde{d_f})\nabla
u\right)$. Clearly, $f\in L^p(0,T;(W^{1,p})')$ for $d_f-\tilde{d_f} \in L^p$
because the remaining terms are uniformly bounded in $L^\infty$. To
apply Gr\"{o}ger's theorem~\cite[Theorem 2.1]{Groger:1992}, we only
need to show the coercivity and boundedness of the operator $A$
defined above. By using the Taylor expansions~\eqref{eqn:taylorz1}
and~\eqref{eqn:taylorgrad1} in the splitting~\eqref{eqn:kdecom}, we
can rearrange the differential $A$ into
\begin{equation*}\begin{aligned}
    \left(Av,w\right) &= \left(k(u,\nabla u;\tilde{d}_f)\left(1-\gamma\right)|\nabla
      \tilde{u}|^{\gamma-1}
      \mathbf{v}\cdot\nabla v\nabla\tilde{u},\nabla w\right)\\
    &-\left(\tilde{d_f}\alpha
      \left(\bar{u}-z\right)^{\alpha-1}|\nabla{\tilde{u}}|^{\gamma-1}v\nabla\tilde{u},\nabla
      w\right) + \left(k(u,\nabla u;\tilde{d}_f)\nabla v,\nabla
      w\right).
\end{aligned}\end{equation*}
In view of the strict positivity of the term $k(u,\nabla
u;\tilde{d}_f)$, that the parameter $\gamma$ is close to one and that the
quantities $u,\,\nabla u$ etc. are uniformly bounded, we deduce that
\begin{equation*}
  \left(Av,v\right)\geq c_A\|\nabla v\|_{L^2} - C_A\|v\|_{L^2}.
\end{equation*}
for some constants $c_A,\,C_A>0$. Hence, the associated matrix-valued
coefficient in the differential operator is pointwise bounded from
below and above away from zero. The continuity of the operator follows
similarly. Consequently, an application of Gr\"{o}ger's
theorem~\cite{Groger:1992} directly yields the desired estimate
$\int_0^T\|v(s)\|_{W^{1,p}}^2ds \leq C \|d\|_{L^\infty}$ for some
$p\in(2,\infty)$.
\end{proof}

\begin{rmk}
  The exponent $p\in(2,\infty)$ in Theorem~\ref{thm:contWp} depends on
  the spatial dimension, the pointwise upper and lower bounds of the
  conductivity $k(u,\nabla u;d_f)$ and the smoothness of the domain
  $\Omega$; see~\cite{Groger:1992} for details.
\end{rmk}

Next we show the boundedness of the linearized map.
\begin{thm}\label{thm:bdd}
  Let the assumptions in Theorem~\ref{thm:cont} be fulfilled, and the
  linear map $F':L^\infty \to L^2(0,T;H^1(\Omega))$ be defined by $d
  \mapsto v$, with $v$ given by
\begin{equation*}
  \begin{aligned}
    \left(v_t,w\right) + &\left(k(u,\nabla u)[I-\left(1-\gamma\right)\tilde{\eta}\otimes\tilde{\eta}]\cdot \nabla v,\nabla w\right)\\
    & + \left(d_f\alpha\frac{\left(u-z\right)^{\alpha-1}v}{|\nabla
        u|^{1-\gamma}}\nabla u,\nabla w\right) =
    -\left(d\frac{\left(u-z\right)^\alpha}{|\nabla u|^{1-\gamma}}\nabla u,\nabla
      w\right),
  \end{aligned}
\end{equation*}
with the initial condition $v(0)=0$. Then the linear map $F'$ is
bounded.
\end{thm}
\begin{proof}
Insert $w := v$ to get
\begin{equation*}
\begin{aligned}
  \tfrac{1}{2}\partial_t\|v\|^2_{L^2(\Omega)} + &\left(k(u,\nabla
  u;d_f)\nabla v, [I-\left(1-\gamma\right)\tilde{\eta}\otimes\tilde{\eta}]
  \nabla v\right)\\
  & = - \left(d_f\alpha\frac{\left(u-z\right)^{\alpha-1}v}{|\nabla
      u|^{1-\gamma}}\nabla u,\nabla v\right)
  -\left(d\frac{\left(u-z\right)^\alpha}{|\nabla u|^{1-\gamma}}\nabla u,\nabla
    v\right):=I+II.
\end{aligned}
\end{equation*}
Using Cauchy-Schwarz inequality and Young's inequality, the term $I$
can be bounded by
\begin{equation*}
  I \le  C(\epsilon_1) \|d_f\|^2_{L^\infty} \|\left(u-z\right)^{\alpha-1} |\nabla u|^\gamma\|^2_{L^\infty} \|v\|^2_{L^2}
  + \epsilon_1\|\nabla v\|^2_{L^2},
\end{equation*}
where $\epsilon_1>0$ is arbitrary. Similarly, the term $II$ can be
bounded by: for any $\epsilon_2>0$
\begin{equation*}
  II\le C(\epsilon_2) \| d \|^2_{L^\infty}  \|d_f\|^2_{L^\infty}  \|\left(u-z\right)^\alpha|\nabla u|^\gamma\|^2_{L^2}
  + \epsilon_2\|\nabla v\|^2_{L^2}.
\end{equation*}
Recall that $\gamma$ is strictly less than unity, and hence
$I-\left(1-\gamma\right)\tilde{\eta}\otimes\tilde{\eta}$ is strictly
positive definite, and the diffusion coefficient $k(u,\nabla u;d_f)$
is strictly positive (independent of $d_f$). Therefore, these
estimates altogether give
\begin{equation*}
  \tfrac{1}{2}\partial_t\|v\|^2_{L^2(\Omega)} + \|\nabla v\|^2_{L^2} \leq C \left(\|v\|^2_{L^2} + \| d \|^2_{L^\infty}\right).
\end{equation*}
Applying Gr\"onwall's inequality and noting that $v(0) = 0$, the
desired assertion follows.
\end{proof}

\begin{rmk}
  The condition that the parameter $\gamma$ is close to $1$ is not
  required in Theorem~\ref{thm:bdd}. A direct application of
  Gr\"{o}ger's theorem indicates that the map $F':L^\infty\mapsto
  L^p(0,T;W^{1,p}(\Omega))$ is also bounded.
\end{rmk}

Finally, we show the Fr\'{e}chet differentiability of the forward map.
\begin{thm}
  Let the assumptions in Theorem~\ref{thm:cont} be fulfilled, and the
  bounded linear map $u'(d_f)d$ be defined in
  Theorem~\ref{thm:bdd}. Then $u'(d_f)d$ is the Fr\'{e}chet derivative
  of the map $d_f \rightarrow u(d_f)$, i.e.,
\begin{equation*}
  \lim_{\|d\|_{L^\infty}\to 0} \frac{\|u(d_f + d)-u(d_f) - u'(d_f)d \|_{L^2(0,T;H^1(\Omega))}}{\|d\|_{L^\infty}} = 0.
\end{equation*}
\end{thm}
\begin{proof}
  We denote by $k(u,\nabla
  u;d_f)=d_f\frac{\left(u-z\right)^\alpha}{|\nabla u|^{1-\gamma}}$ and
  $\tilde{d}_f=d_f+d$, $\tilde{u}=u(\tilde{d}_f)$, $u = u(d_f)$ and
  $\bar{u}=u'(d_f)d$ and let $v=\tilde{u}-u$,
  $w=\tilde{u}-u-\bar{u}$. We also denote
  $D(u)=\left(1-\gamma\right)\tilde{\eta}\otimes\tilde{\eta}$. Then
  it directly follows from the weak formulations for $\tilde{u}$, $u$
  and $\bar{u}$ that
\begin{equation*}
  \begin{aligned}
    \left(w_t,w\right) + &\left(k(\tilde{u},\nabla
        \tilde{u};\tilde{d}_f)\nabla \tilde{u}
      -k(u,\nabla u;\tilde{d}_f)\nabla u
      ,\nabla w\right)\\
    =& \left(k(u,\nabla u;d_f)\left(I-D(u)\right) \nabla
          \bar{u},\nabla w\right) +
        \left(d_f\alpha\frac{\left(u-z\right)^{\alpha-1}\bar{u}}{|\nabla
            u|^{1-\gamma}}\nabla u,\nabla w\right),
  \end{aligned}
\end{equation*}
which upon rearrangement and noting the assumptions on $u$ and $\nabla
u$ yields
\begin{equation*}
  \begin{aligned}
    (w_t,w) + (k(u,\nabla u;\tilde{d}_f)\nabla w,\nabla w)
    =&\underbrace{((k(u,\nabla
      u;\tilde{d}_f)-k(\tilde{u},\nabla \tilde{u};\tilde{d}_f))\nabla
      \tilde{u},\nabla w)}_{I}
    \underbrace{- (k(u,\nabla u;d)\nabla\bar{u},\nabla w)}_{II}\\
    & - \underbrace{(k(u,\nabla u;d_f)D(u) \nabla\bar{u},\nabla
      w)}_{III}
    +\underbrace{\left(d_f\alpha\frac{(u-z)^{\alpha-1}\bar{u}}{|\nabla
          u|^{1-\gamma}}\nabla u,\nabla w\right)}_{IV}.
  \end{aligned}
\end{equation*}
It suffices to estimate the four terms on the right hand side. First,
by means of Cauchy-Schwarz inequality and Young's inequality, the term
$II$ can be estimated by
\begin{equation*}
   \begin{aligned}
      |II| \leq & C\|d\|_{L^\infty}\|\nabla \bar{u}\|_{L^2}\|\nabla w\|_{L^2},
   \end{aligned}
\end{equation*}
To bound the first term $I$, we further split it into
\begin{equation*}
  \begin{aligned}
    I=\left(k(u,\nabla u;\tilde{d}_f)\frac{(
        |\nabla{\tilde{u}}|^{1-\gamma} - |\nabla{u}|^{1-\gamma})}
      {|\nabla{\tilde{u}}|^{1-\gamma}}\nabla\tilde{u},\nabla w\right)
    +\left(\tilde{d}_f\frac{(u-z)^\alpha -
        (\tilde{u}-z)^\alpha}{|\nabla{\tilde{u}}|^{1-\gamma}}\nabla\tilde{u},\nabla
      w\right):=V+VI.
  \end{aligned}
\end{equation*}
Now we employ the Taylor expansion
\begin{equation*}
  |\nabla \tilde{u}|^{1-\gamma} = |\nabla u|^{1-\gamma}+(1-\gamma)|\nabla u|^{-\gamma-1}\nabla u\cdot\nabla v
  + \boldsymbol{K}\nabla v^2
\end{equation*}
with the matrix-valued function $\boldsymbol{K}$ given by
\begin{equation*}
  \boldsymbol{K}=-\int_0^1(1-t)\left((1-\gamma^2)\boldsymbol{\phi}(t)\boldsymbol{\phi}(t)^\mathrm{t}|\boldsymbol{\phi}(t)|^{-\gamma-3}
    + (1-\gamma)|\boldsymbol{\phi}(t)|^{-1-\gamma}\boldsymbol{I} \right)dt
\end{equation*}
and $\boldsymbol{\phi}(t)=\nabla (u + tv)$. With the help of this
expansion, we derive that
\begin{equation*}\begin{aligned}
    V-III &=(k(u,\nabla u;d_f)D(u) \nabla w,\nabla w)+
    (1-\gamma)\underbrace{ \left(k(u,\nabla u;d_f)\frac{\nabla
          u\cdot\nabla v}{|\nabla u|} \left(\frac{|\nabla
            u|^{1-\gamma}\nabla\tilde{u}}{|\nabla
            \tilde{u}|^{1-\gamma}|\nabla u|}
          -\frac{\nabla u}{|\nabla u|}\right),\nabla w\right)}_{VII}\\
    &+(1-\gamma)\underbrace{\left(k(u,\nabla u;d)\frac{\nabla
          u\cdot\nabla v}{|\nabla u|^{1+\gamma}}
        \frac{\nabla\tilde{u}}{|\nabla \tilde{u}|^{1-\gamma}},\nabla
        w\right)}_{VIII} + \underbrace{\left(k(u,\nabla
        u;\tilde{d}_f)\frac{\boldsymbol{K}\nabla v^2}{|\nabla
          \tilde{u}|^{1-\gamma}} \nabla\tilde{u},\nabla w\right)}_{IX} \\
\end{aligned}\end{equation*}
Next we estimate the terms on the right-hand side one by one. First,
let $p$ be the exponent from theorem~\ref{thm:contWp} and choose $q>2$
such that $\frac{1}{p} + \frac{1}{q} = \frac{1}{2}$. Then by the
uniform $L^\infty$ boundness of $u$ and $\nabla u$ (also $\tilde{u}$,
$\nabla \tilde{u}$ etc.)
\begin{equation*}
  \begin{aligned}
    VII&=\left(k(u,\nabla u;d_f)\frac{\nabla u\cdot \nabla
        v}{|\nabla u|^2}|\nabla\tilde{u}|^{\gamma-1}
      \left(|\nabla u|^{1-\gamma}\nabla v + \nabla u(|\nabla u|^{1-\gamma}-|\nabla \tilde{u}|^{1-\gamma})\right),\nabla w\right)\\
    &\le C \|\nabla v\|_{L^p}\|\nabla v\|_{L^q}\|\nabla w\|_{L^2}
    + C \||\nabla v| (|\nabla u|^{1-\gamma}-|\nabla \tilde{u}|^{1-\gamma})\|_{L^2}\|\nabla w\|_{L^2}\\
    & \le C \|\nabla v\|_{L^p}\|\nabla v\|_{L^q}\|\nabla w\|_{L^2}
    + C \||\nabla v|^2\|_{L^2} \|\nabla w\|_{L^2}\\
    &\le C \|\nabla v\|_{L^p}\|\nabla v\|_{L^q}\|\nabla w\|_{L^2}\le C
    \|\nabla v\|^{1+\delta}_{L^p}\|\nabla w\|_{L^2},
  \end{aligned}
\end{equation*}
where in the third line we have utilized the
expansion~\eqref{eqn:taylorgrad1}, and the last line follows from the
fact that either $\|\nabla v\|_{L^q} \le C \|\nabla v\|_{L^p}$ holds
for $q<p$ or $\|\nabla v\|_{L^q} \le C \|\nabla v\|^\delta_{L^p}$
holds for some $0 < \delta < 1$ due to the $L^\infty$-boundedness of
$\nabla u$ and $\nabla\tilde u$.  Similarly, the terms $VIII$ and $IX$
can be bounded by
\begin{equation*}
  VIII \le C\|d\|_{L^\infty}\|\nabla v\|_{L^2}\|\nabla w\|_{L^2}\quad \mbox{and}\quad
  IX \le C\|\nabla v\|^{1+\delta}_{L^p}\|\nabla w\|_{L^2}.
\end{equation*}
Next we combine the terms $VI$ and $IV$. To this end, we employ the
Taylor expansion
\begin{equation*}
  (\tilde{u} - z)^\alpha = (u - z)^\alpha + \alpha (u - z)^{\alpha-1} v
  +\tfrac{1}{2} \alpha(\alpha-1) (\hat{u}-z )^{\alpha-2} v^2
\end{equation*}
with $\hat{u}$ being some function pointwise between $u$ and
$\tilde{u}$ we can estimate.  With the help of this identity, we
arrive at the following splitting
\begin{equation*}\begin{aligned}
    VI + IV &= -\left(d_f\alpha\frac{(u-z)^{\alpha-1} w}{|\nabla
        u|^{1-\gamma}}\nabla u,\nabla w\right) +
    \underbrace{\left(d_f\alpha(u-z)^{\alpha-1}v\left(\frac{\nabla
            u}{|\nabla u|^{1-\gamma}}-
          \frac{\nabla\tilde{u}}{|\nabla{\tilde{u}}|^{1-\gamma}}\right),\nabla w\right)}_X\\
    &- \underbrace{\left(d\alpha
        \frac{(u-z)^{\alpha-1}}{|\nabla{\tilde{u}}|^{1-\gamma}}v\nabla\tilde{u},\nabla
        w\right)}_{XI} -\underbrace{\tfrac{1}{2}\left(\tilde{d}_f
        \alpha(\alpha-1) \frac{(\eta -z
          )^{\alpha-2}}{|\nabla{\tilde{u}}|^{1-\gamma}} v^2 \nabla
        \tilde{u},\nabla w\right)}_{XII}
\end{aligned}
\end{equation*}
Consequently, by the uniform boundedness of the quantities $u$,
$\nabla u$ (and $\tilde{u}$, $\nabla\tilde{u}$ etc.) and Sobolev
embedding theorem, we have
\begin{equation*}
  \begin{aligned}
    X \leq C\|v\|_{W^{1,P}}^{1+\delta}\|\nabla w\|_{L^2},\quad XI \le
    C \|d\|_{L^\infty}\|v\|_{W^{1,p}} \|\nabla w\|_{L^2},\quad XII
    \leq C \|\nabla v\|^{1+\delta}_{W^{1,p}}\|\nabla w\|_{L^2}.
  \end{aligned}
\end{equation*}
These estimates, Young's inequality and that $\gamma$ is close to
unity (hence $D(u)$ can be made arbitrarily small for $\gamma$ close
to unity) yield
\begin{equation*}
  \tfrac{1}{2}\partial_t\|w\|^2_{L^2} + \tfrac{C_K}{2}\|\nabla w\|^2_{L^2}\leq
  C\left(\|d\|_{L^\infty}^2\|v\|_{W^{1,p}}^2+\|v\|_{W^{1,p}}^{2+2\delta}+\|w\|_{L^2}^2\right).
\end{equation*}
Finally, an application of Gr\"onwall's inequality and
Theorem~\ref{thm:contWp} lead to
\begin{equation*}
  \|w\|^2_{L^2} + \int_0^T\|\nabla w\|^2_{L^2}ds \le C \|d\|^{2+2\delta}_{L^\infty}
\end{equation*}
upon noting the initial condition $w(0) =0$. This concludes the proof.
\end{proof}

\begin{rmk}
An inspection of the proof indicates that the assumptions on the 
solution $u(d_f)$ and gradient can be greatly relaxed 
if the parameter $\gamma=1$. The latter case is analogous to 
the porous media equation, and thus the results are of independent interest.
\end{rmk}
\section{Generalized-$\alpha$ method}\label{app:gam}
In this appendix, we describe the generalized-$\alpha$ method. Note
that for the full discretization of the forward problem, each time
step involves solving a highly nonlinear (and possibly also stiff)
system. Hence a careful treatment of the time stepping is required. To
this end, we employ the so-called generalized-$\alpha$ method together
with a predictor-corrector
method~\cite{JansenWhitingHulbert:2000,ColierRadwanDalcinCalo:2011}. For
a first-order system, the method can be stated as follows: given
$(u_n,\dot{u}_n)$, find
$(u_{n+1},\dot{u}_{n+1},u_{n+\alpha_f},\dot{u}_{n+\alpha_m})$ such
that
\begin{equation*}
  \left\{
  \begin{aligned}
    R(u_{n+\alpha_f},\dot{u}_{n+\alpha_m}) & = 0,\\
    u_{n+\alpha_f} &= u_n + \alpha_f(u_{n+1}-u_n),\\
    \dot{u}_{n+\alpha_m} &= \dot{u}_{n} + \alpha_m(\dot{u}_{n+1}-\dot{u}_{n}),\\
    u_{n+1} & = u_n + \Delta t ((1-\gamma)\dot{u}_n+\gamma
    \dot{u}_{n+1}),
  \end{aligned}\right.
\end{equation*}
where $\Delta t= t_{n+1}-t_n$ is the time step size, $\alpha_f$,
$\alpha_m$ and $\gamma$ are real valued parameters of the method, and
$R(u_{n+\alpha_f},\dot{u}_{n+\alpha_m})$ denotes the (discrete)
residual of the nonlinear system. For a linear model problem,
unconditional stability of the scheme is attained if $\alpha_m\geq
\alpha_f\geq\frac{1}{2}$, and a second-order accuracy can be achieved
with the choice $\gamma=\frac{1}{2}
+\alpha_m-\alpha_f$~\cite{JansenWhitingHulbert:2000}. The method can
be succinctly parameterized by the spectral radius $\rho_\infty$ into
a one-parameter family. Then the parameters $\alpha_m$, $\alpha_f$ and
$\gamma$ can be expressed as~\cite{JansenWhitingHulbert:2000}
\begin{equation*}
  \alpha_f = \frac{1}{1+\rho_\infty},\quad \alpha_m=\frac{3-\rho_\infty}{2(1+\rho_\infty)}, \quad \gamma=\frac{1}{1+\rho_\infty}.
\end{equation*}
A complete description of the generalized-$\alpha$ method is given in
Algorithm~\ref{alg:gam}. It is of predictor/corrector type with
correctors computed by a Newton method, where the superscript indices
indicate the corrector steps within the loop. In our implementation,
we have set $\rho_\infty=0.1$, and the tolerance $\epsilon$ in the
stopping criterion to $1.0\times10^{-6}$ and the maximum number of
iterations ($\mathrm{MaxIter}$) to $20$. The major computational
effort of Algorithm~\ref{alg:gam} lies in calculating the Jacobian
matrix $K_{n+1}^{(i)}$ for the Newton system, i.e., step $5$. For
large-scale problems, iterative solvers, e.g., GMRES or BiCGstab,
which requires only matrix-vector multiplication, are
preferable~\cite{JansenWhitingHulbert:2000}.

\begin{algorithm}[t]
	\caption{Generalized-$\alpha$ method.\label{alg:gam}}
\begin{algorithmic}[1]
  \STATE Compute predictor $u_{n+1}^{(0)}=u_n$ and
  $\dot{u}_{n+1}^{(0)}=\frac{\gamma-1}{\gamma}\dot{u}_n$, and set
  $i=0$.  \STATE Set the initial guess of $u_{n+\alpha_f}^{(0)}$ and
  $u_{n+\alpha_m}^{(0)}$ as
    \begin{equation*}
      u_{n+\alpha_f}^{(0)}=u_n+\alpha_f(u_{n+1}^{(0)}-u_n)\quad \mbox{and}\quad
      \dot{u}_{n+\alpha_m}^{(0)}=\dot{u}_n+\alpha_m(\dot{u}_{n+1}^{(0)}-\dot{u}_n).
           \end{equation*}
           \WHILE {$i<\mathrm{MaxIter}$} \STATE Evaluate the Newton
           residual $\displaystyle R_{n+1}^{(i)} =
           R(u_{n+\alpha_f}^{(i)},\dot{u}_{n+\alpha_m}^{(i)})$.
           \STATE Calculate the Jacobian
          \begin{equation*}
            K_{n+1}^{(i)}=\frac{\partial R(u_{n+\alpha_f}^{(i)},\dot{u}_{n+\alpha_m}^{(i)})}{\partial u_{n+\alpha_f}}
            +\alpha_m[\alpha_f\gamma\Delta t]^{-1}\frac{\partial R(u_{n+\alpha_f}^{(i)},
              \dot{u}_{n+\alpha_m}^{(i)})}{\partial \dot{u}_{n+\alpha_m}}.
          \end{equation*}
          \STATE Solve Newton system for the corrector $\Delta
          u_{n+1}^{(i)}$ from $K_{n+1}^{(i)}\Delta
          u_{n+1}^{(i)}=-R_{n+1}^{(i)}$.  \STATE Update the solutions
          $u_{n+\alpha_f}^{(i+1)}$ and $\dot{u}_{n+\alpha_m}^{(i+1)}$
          by
          \begin{equation*}
             \begin{aligned}
               u_{n+\alpha_f}^{(i+1)}&=u_{n+1}^{(i)} + \Delta u_{n+1}^{(i)},\\
               \dot{u}_{n+\alpha_m}^{(i+1)}&=(1-\gamma^{-1}\alpha_m)\dot{u}_{n+1}^{(i)}
               +\alpha_m[\gamma\Delta t\alpha_f]^{-1}(u_{n+\alpha_f}^{(i+1)}-u_n).
             \end{aligned}
          \end{equation*}
       \STATE Check the stopping criterion: if {$\displaystyle
            \|R_{n+1}^{(i)}\|\leq\epsilon\|R_{n+1}^{(0)}\|$}, stop iteration.
       \STATE Increase index $i=i+1$.
    \ENDWHILE
       \STATE Output the solutions $u_{n+1}$ and $\dot{u}_{n+1}$ by
           \begin{equation*}
               u_{n+1} = u_n+\alpha_f^{-1}(u_{n+\alpha_f}^{(\mathrm{MaxIter})}-u_n) \quad \mbox{and}\quad
               \dot{u}_{n+1} =\dot{u}_n + \alpha_m^{-1}(\dot{u}_{n+\alpha_m}^{(\mathrm{MaxIter})}-\dot{u}_n).
           \end{equation*}
\end{algorithmic}
\end{algorithm}

\bibliographystyle{abbrv}
\bibliography{dsw}
\end{document}